\documentclass[12pt]{amsart}
\usepackage[english]{babel}
\usepackage[utf8x]{inputenc}
\usepackage[T1]{fontenc}

\usepackage[margin=1in]{geometry}

\usepackage[alphabetic]{amsrefs}

\usepackage{amsmath}
\usepackage{amstext}
\usepackage{amsfonts}
\usepackage{amssymb}
\usepackage{amsthm}
\usepackage{amsrefs}
\usepackage{mathtools}
\usepackage{dsfont}
\usepackage{color}
\usepackage{graphicx}
\usepackage[colorinlistoftodos]{todonotes}
\usepackage[colorlinks=true, allcolors=blue]{hyperref}
\usepackage{float}
\usepackage{mathrsfs}
\allowdisplaybreaks

\newtheorem{theorem}{Theorem}[section]

\newtheorem{prop}[theorem]{Proposition}

\newtheorem{cor}[theorem]{Corollary}
\newtheorem{theo}[theorem]{Theorem}
\newtheorem{lem}[theorem]{Lemma}

\newtheorem{thmintro}{Theorem}

\newtheorem*{cor*}{Corollary}
\newtheorem*{theo*}{Theorem}
\newtheorem*{lem*}{Lemma}

\newtheorem*{prop*}{Proposition}

\theoremstyle{definition}
\newtheorem{definition}[theorem]{Definition}

\newtheorem{ex}[theorem]{Example}

\newtheorem*{defn*}{Definition}

\theoremstyle{remark}

\newcommand{\cG}{\mathcal{G}}

\newcommand{\cJ}{\mathcal{J}}

\newcommand{\cS}{\mathcal{S}}

\newcommand{\cX}{\mathcal{X}}

\newcommand{\acts}{\curvearrowright}

\newcommand{\Hi}{\mathcal{H}}

\newcommand{\g}{\gamma}

\newcommand{\C}{\mathbb C}

\usepackage[nodisplayskipstretch]{setspace}
\allowdisplaybreaks

\title[]{The action of an inverse semigroup on its Stone-\v Cech compactification}

\author[J. P. Z. Gondek]{Joseph P. Z. Gondek}
	\address[J. P. Z. Gondek]{University of Houston, Department of Mathematics. Houston, TX,
77204-3008, United States
}
	\email{jpgondek@cougarnet.uh.edu}
\author[C. Starling]{Charles Starling}
	\address[C. Starling]{Carleton University, School of Mathematics and Statistics. 4302 Herzberg Laboratories}
	\email{cstar@math.carleton.ca}
\date{}

\begin{document}

\begin{abstract}
We initiate the study of the Stone-\v Cech transformation groupoid $\mathcal{G} = \mathcal{S}\ltimes\beta\mathcal{S}$ of an inverse semigroup $\mathcal{S}$. We prove that the properties of being Hausdorff, principal, and effective are all equivalent for $\mathcal{G}$, and give an algebraic condition on $\mathcal{S}$ equivalent to the Hausdorffness of $\mathcal{G}$. We show that the Hausdorffness of Exel's tight groupoid $\mathcal{G}_{\text{tight}}(\mathcal{S})$ is necessary for the Hausdorffness of $\mathcal{G}$. Finally, we clarify the connection between several crossed product constructions involving this groupoid, and show that when $\mathcal{G}$ is Hausdorff and $\mathcal{S}$ has the Property (FL) of Lled\'o and Mart\'inez, then $\mathcal{G}$ is amenable if and only if the reduced C$^*$-algebra $\textup{C}^*_r(\mathcal{S})$ is exact.
\end{abstract}
\maketitle

\section{Introduction} 

\let\thefootnote\relax\footnotetext{2020 \textit{Mathematics Subject Classification} Primary: 20M18, 18B40, 46L05; Secondary: 43A65.\\[1ex] \textit{Keywords}: \'etale groupoids, inverse semigroups, Stone-\v Cech compactification.\\[1ex] Part of this work was completed as part of the Master's thesis of the first author. The second author is funded by NSERC grant RGPIN-2021-03834.}

A semigroup $\mathcal{S}$ is called \textit{inverse} if, for every $s\in\mathcal{S}$, there is a \textbf{unique} $s^*\in\mathcal{S}$ with $ss^*s = s$ and $s^*ss^* = s^*$. There have been recent efforts to understand the dynamics and representation theory of inverse semigroups in a way that generalizes the discrete group setting. The first to investigate the exactness of $\textup{C}^*_r(\mathcal{S})$ seems to be Anantharaman-Delaroche in the monograph \cite{adeg}, where groupoid C$^*$-algebra techniques are applied to show that $\textup{C}^*_r(\mathcal{S})\subset\mathcal{B}(\ell^2(\mathcal{S}))$ is exact whenever the maximum group homomorphic image $\mathcal{S}{/}\sigma$ is an exact group. In \cite{lledomartinez}, Lled\'o and Mart\'inez construct and study the uniform Roe algebra $\mathcal{R}_{\mathcal{S}}\subset \mathcal{B}(\ell^2(\mathcal{S}))$ of a countable inverse semigroup $\mathcal{S}$. Under a mild combinatorial assumption (admitting a finite labelling), they show that the reduced inverse semigroup C$^*$-algebra $\textup{C}^*_r(\mathcal{S})$ is exact if and only if $\mathcal{R}_\mathcal{S}$ is nuclear.

In this note, we consider the contravariant action $\mathcal{S}\acts \beta\mathcal{S}$, and investigate the properties of the resulting action groupoid $\mathcal{S}\ltimes\mathcal{\beta}\mathcal{S}$ and reduced C$^*$-algebra $\textup{C}^*_r(\mathcal{S}\ltimes\mathcal{\beta}\mathcal{S}) \cong \ell^{\infty}(\mathcal{S})\rtimes_r\mathcal{S}$. We first consider the question of effectiveness and principality for $\mathcal{S}\ltimes \beta\mathcal{S}$. These conditions always hold in the group case, and correspond to the fact that the action $\Gamma\acts \beta\Gamma$ is always free (see for example \cite{NR14}*{Example~2.14}). In contrast to the group case, and somewhat surprisingly, there are inverse semigroups $\mathcal{S}$ such that $\mathcal{S}\ltimes \beta\mathcal{S}$ is not effective nor principal. Using a finite cover criterion of Exel and Pardo \cite{EP16}, we obtain the following generalization of the group case:

\begin{thmintro}\label{mainthm:B}
    Let $\cS$ be an inverse semigroup. Then the following statements are equivalent:
    \begin{enumerate}
        \item $\cS\ltimes\beta\cS$ is Hausdorff;
        \item $\cS\ltimes\beta\cS$ is effective;
        \item $\cS\ltimes\beta\cS$ is principal;
        \item Every clopen action of $\cS$ on a locally compact Hausdorff space has a Hausdorff transformation groupoid.
    \end{enumerate}
\end{thmintro}

This generalizes the group case because $\mathcal{S}\ltimes\beta\cS $ is always Hausdorff when $\cS$ is a group. Using Theorem~\ref{mainthm:B}, we give an algebraic criterion on $\mathcal{S}$ (Theorem \ref{th:Hausdorff_equality}) which implies $\mathcal{S}\ltimes\beta\cS$ is Hausdorff for the class of abstract pseudogroups (Corollary~\ref{hausdorffexample}). Finally, since there are many examples of inverse semigroups with non-Hausdorff tight groupoids, condition (3) of Theorem~\ref{mainthm:B} provides many examples of inverse semigroups for which $\cS\ltimes\beta\cS$ is not Hausdorff nor principal.

To consider the amenability of the action $\mathcal{S}\acts\beta\mathcal{S}$ in the sense of Exel and Starling \cite{starling}, we clarify the relationship between the uniform Roe algebra $\mathcal{R}_{\mathcal{S}}$ and the reduced crossed product $\ell^{\infty}(\mathcal{S})\rtimes_r\mathcal{S}$ in the sense of \cite{bussexel1}. In \cite{lledomartinez}, the uniform Roe algebra $\mathcal{R}_{\mathcal{S}}\subset \mathcal{B}(\ell^2(\mathcal{S}))$ is shown to embed into $\mathcal{B}(\ell^2(\mathcal{S})\otimes\ell^2(\mathcal{S}))$ as the image of a covariant representation of the action $\mathcal{S}\acts\ell^{\infty}(\mathcal{S})$. It follows from Theorem~\ref{mainthm:B} that this C$^*$-algebra, denoted in \cite{lledomartinez} by ``$\ell^{\infty}(\mathcal{S})\rtimes_r\mathcal{S}$'' and in this note by $\ell^{\infty}(\mathcal{S})\rtimes_{rr}\mathcal{S}$, is not to be confused (in full generality) with the reduced crossed product $\ell^{\infty}(\mathcal{S})\rtimes_r\mathcal{S}$ in the sense of \cite{bussexel1}. However, these C$^*$-algebras often coincide:

\begin{thmintro}\label{mainthm:C}
The following conditions are equivalent for an inverse semigroup $\mathcal{S}$:
\begin{enumerate}
\item $\mathcal{S}\ltimes\beta\mathcal{S}$ is Hausdorff;
\item $\ell^{\infty}(\mathcal{S})\rtimes_{rr}\mathcal{S}\cong \ell^{\infty}(\mathcal{S})\rtimes_r\mathcal{S}$ canonically.
\end{enumerate}
\end{thmintro}

As a consequence, we obtain a generalization of a discrete group result.

\begin{thmintro}\label{mainthm:D}
Suppose $\mathcal{S}$ is a countable inverse semigroup with the property \textup{(FL)} of Lled\'o and Mart\'inez. If $\mathcal{S}\ltimes\beta\mathcal{S}$ is Hausdorff, then the action $\mathcal{S}\acts\beta\mathcal{S}$ is amenable \textup{(}in the sense of \cite{starling}\textup{)} if and only if $\textup{C}^*_r(\mathcal{S})$ is exact.
\end{thmintro}

We would like to thank the anonymous referee for several suggestions and corrections from the previous version to this manuscript, and particularly for pointing out how our results in that version suggest the content of Theorem \ref{mainthm:C}.

\section{Preliminaries}

In what follows, $\mathcal{S}$ denotes a countable inverse semigroup. We refer the reader to Lawson's book \cite{L} for a detailed treatment of these semigroups. A subset $A\subset\mathcal{S}$ is called \textit{symmetric} if $s^*\in A$ whenever $s\in A$. An element $e\in\mathcal{S}$ is called \textit{idempotent} if $e^2 = e$. For any $s\in\mathcal{S}$, $s^*s$ and $ss^*$ are idempotent. Denote by $E_{\mathcal{S}}$ the set of idempotent elements of $\mathcal{S}$. It is a pleasant fact that $E_{\mathcal{S}}$ is a commutative subsemigroup of $\mathcal{S}$, and that $\mathcal{S}$ is a group if and only if $E_{\mathcal{S}}$ is a singleton. If $X$ is a set, then the collection $\mathcal{I}(X)$ of all partial bijections between subsets of $X$ is an inverse semigroup under the operation of partial composition. The semigroup $\mathcal{I}(X)$ is called the \textit{symmetric inverse monoid} on $X$.

Inverse semigroups carry a natural partial order: $s\leqslant t$ if and only if $ts^*s =s$. Restricted to $E_\cS$, this order is given by $e\leqslant f$ if and only if $ef = e$. 

For any poset $(P,\prec)$ and subset $A\subseteq P$ we will use the notation
\[
A^\prec : = \{b\in P: a\prec b\text{ for some }a\in A\}.
\]
In this paper, we only consider the partial orders $\subseteq$ on the power set of $\cS$, and $\leqslant$ or $\geqslant$ on $\cS$.

An \textit{action} of $\mathcal{S}$ on a compact Hausdorff space $\mathcal{X}$ is a triple $(\mathcal{S}, \mathcal{X}, \alpha)$, where $\alpha: \mathcal{S}\to \mathcal{I}(\mathcal{X})$ is a semigroup homomorphism such that, \hbox{for all $s\in\mathcal{S}$,}
\begin{enumerate}
\item $\alpha_s$ is continuous, for all $s\in\mathcal{S}$;
\item $\text{dom}(\alpha_s)$ is open, for all $s\in\mathcal{S}$.
\end{enumerate}
One can see that the domain of $\alpha_s$ is the same as that of $\alpha_{s^*s}$. For each $e\in E_\cS$ we write $D_e^\alpha$ for the domain of $\alpha_e$ (or simply $D_e$ if the action is understood), so that $\alpha_s: D_{s^*s}^\alpha\to D_{ss^*}^\alpha$. The main action we are concerned with in this paper is described below.

\begin{ex}[\cite{gondek24}, Section 3.2]\label{sc} We define the canonical action of an inverse semigroup $\mathcal{S}$ on its Stone-\v Cech compactification $(\beta \mathcal{S}, \iota)$, where $\mathcal{S}$ is viewed as a discrete space and $\iota: \mathcal{S}\to \beta \mathcal{S}$ is the canonical embedding. For each $s\in\mathcal{S}$, define a bijection $\lambda_s: s^*\mathcal{S}\to s\mathcal{S}$ by $\lambda_st = st$ ($t\in s^*\mathcal{S}$). By the universal property of $\beta\cS$, $\lambda_s$ extends to a homeomorphism $\lambda_s: \overline{s^*\mathcal{S}}\to \overline{s\mathcal{S}}$, where $\overline{s^*\mathcal{S}} := \overline{\iota(s^*\mathcal{S})}^{\beta}\subseteq \beta \mathcal{S}$. We observe that the domain and codomain of $\lambda_s$ are always compact, open subsets of $\beta\mathcal{S}$. It is routine to check that the mapping $\mathcal{S}\ni s\mapsto \lambda_s\in\mathcal{I}(\beta\mathcal{S})$ is a semigroup homomorphism, and therefore an action, which we shall henceforth \hbox{refer to by $\mathcal{S}\acts\beta\mathcal{S}$.}
\end{ex}

The appropriate way to interpret the dynamics of an inverse semigroup action $\mathcal{S}\acts\mathcal{X}$ is through the corresponding \textit{transformation groupoid}, denoted by $\mathcal{S}\ltimes\mathcal{X}$. Before reminding the reader of this construction, we establish our conventions for \'etale groupoids. We direct the reader to the lecture notes \cite{simsNotes} of Sims for details. A {\em groupoid} is a small category in which every arrow is invertible. Given a groupoid $\mathcal{G}$, we will denote by $\mathcal{G}^{(2)}\subset\mathcal{G}\times\mathcal{G}$ the composable pairs, and $\mathcal{G}^{(0)} = \{\gamma^{-1}\gamma: \gamma\in \mathcal{G}\}$ the unit space. The \textit{source} and \textit{range} maps $\mathbf{d}: \mathcal{G}\to \mathcal{G}^{(0)}$ and $\mathbf{r}: \mathcal{G}\to \mathcal{G}^{(0)}$ (respectively) are defined by $\mathbf{d}(\gamma) = \gamma^{-1}\gamma$ and $\mathbf{r}(\gamma) = \gamma\gamma^{-1}$, for $\gamma\in\mathcal{G}$. Given $u\in \mathcal{G}^{(0)}$, we will write $\mathcal{G}_{u} = \mathbf{d}^{-1}(u)$, $\mathcal{G}^{u} = \mathbf{r}^{-1}(u)$, and $\mathcal{G}^u_u = \mathcal{G}_u\cap \mathcal{G}^u$. The \textit{isotropy} $\textup{Iso}(\mathcal{G})$ is defined by
\begin{align*}
\textup{Iso}(\mathcal{G}) = \{\gamma\in\mathcal{G}: \textbf{r}(\gamma) = \textbf{d}(\gamma)\}.
\end{align*}
We recall that the groupoid $\mathcal{G}$ is called \textit{principal} if $\textup{Iso}(\mathcal{G}) = \mathcal{G}^{(0)}$.

A \textit{topological groupoid} is a groupoid $\mathcal{G}$ equipped with a topology under which the multiplication $\mathcal{G}^{(2)}\ni (\alpha, \beta)\mapsto \alpha\beta$ and inverse map $\mathcal{G}\ni \gamma\mapsto \gamma^{-1}$ are continuous. A topological groupoid is called \textit{\'etale} if it is locally compact, $\mathcal{G}^{(0)}$ is Hausdorff in the subspace topology inherited from $\mathcal{G}$, and the maps $\mathbf{d}$ and $\mathbf{r}$ are local homeomorphisms. A \textit{bisection} is a subset $U\subset\mathcal{G}$ such that $\mathbf{d}|_{U}$ and $\mathbf{r}|_{U}$ are local homeomorphisms. We remind the reader that \'etale groupoids are not required to be Hausdorff. It is a standard fact that an \'etale groupoid $\mathcal{G}$ is Hausdorff if and only if $\mathcal{G}^{(0)}$ is closed.

Now let $(\mathcal{S}, \mathcal{X}, \alpha)$ be an inverse semigroup action, where $\mathcal{X}$ is a locally compact Hausdorff space. We associate an \'etale groupoid to $(\mathcal{S}, \mathcal{X}, \alpha)$ as follows (see \cite{EXEL} for details). Set
\begin{align*}\Omega = \{(s, x): x\in D_{s^*s}\}\subseteq\mathcal{S}\times \mathcal{X}.
\end{align*}
We define an equivalence relation on $\Omega$ by 
\begin{align*}
(s, x)\sim (t, y)\Leftrightarrow x = y \text{ and } \exists e\in E_{\mathcal{S}} \text{ with } x\in D_e \text{ and } se = te.
\end{align*}
One defines a groupoid structure on $\mathcal{G} := \Omega/{\sim}$ as follows:
\begin{enumerate}
\item $\mathcal{G}^{(2)} = \{([s, x], [t, y])\in\mathcal{G}\times\mathcal{G}: x = \alpha_{t}(y)\}$;
\item $[s, \alpha_t(x)][t, x] = [st, x]$;
\item $[s, x]^{-1} = [s^*, \alpha_s(x)]$.
\end{enumerate}
Equipped with the quotient topology, $\mathcal{G}$ becomes an \'etale groupoid, with basic open neighbourhoods of the form $\Theta(s, U) = \{[s, x]: x\in U\}$ (for $s\in\mathcal{S}$ and open $U\subseteq\mathcal{X}$). We write $\mathcal{G} = \mathcal{S}\ltimes_{\alpha}\mathcal{X}$, and call $\mathcal{G}$ the \textit{transformation groupoid} (or \textit{groupoid of germs}) of the action $\mathcal{S}\acts_{\alpha} \mathcal{X}$.
\section{The Stone-\v Cech transformation groupoid of an inverse semigroup}

Throughout this section, we identify $\beta\mathcal{S}$ with the space of ultrafilters on the discrete space $\mathcal{S}$. Given $A\subseteq \cS$ we let
\begin{align*}U_A := \{\xi\in\beta\mathcal{S}: A\in\xi\}.
\end{align*}
Sets of this form generate the topology on $\beta\cS$. Each element $s\in \cS$ determines a principal ultrafilter $\{A\subseteq \cS: s\in A\} = \{s\}^\subseteq$, so the canonical embedding described in Example~\ref{sc} is given by $\iota(s) = \{s\}^\subseteq$. In what follows we will drop the $\iota$ and consider $\cS$ as a subset of $\beta\cS$. We then have that $\overline{\cS} = \beta\cS$. Now observe that
\[
U_A = \overline{A};
\]
that is, the basic open set $U_A$ is the closure of $A$ in $\beta\cS$. To see this, note that $\xi\in \overline{A}$ if and only if, for all $B\in\xi$, there is some $a\in A$ with $a\in B$. This condition forces $A\in \xi$, because $\xi$ is an ultrafilter and otherwise one would have $\mathcal{S}\backslash A\in\xi$. Conversely, if $A\in\xi$, then already $A\cap B\neq\emptyset$ for all $B\in\xi$, so $\xi\in \overline{A}$.

Since $\mathcal{S}$ is a discrete space, there is a canonical homeomorphism $f: \beta\mathcal{S}\supset\overline{A}\to\beta A$, for all $A\subset\mathcal{S}$. With ultrafilters, this map is defined by $f(\xi) = \{B\in\xi: B\subseteq A\}$.

We now turn to studying the action of $\cS$ on its Stone-\v Cech compactification and properties of its transformation groupoid $\cS\ltimes\beta\cS$. In \cite{EP16}, Exel and Pardo proved many general results about such actions and groupoids, so we recall some terminology and facts from there. 

For $s\in \mathcal{S}$ we let
\begin{align*}
\cJ_s&:= \{e\in E_\cS: se = e\},\\
F_s &:= \{x\in \overline{s^*\mathcal{S}}: \lambda_s(x) = x\},\\
TF_s &:= \bigcup_{e\in \cJ_s}\overline{e\mathcal{S}}.
\end{align*}
Following the terminology given in \cite{EP16}, the elements of $F_s$ are called the {\em fixed points} of $s$, and the elements of $TF_s$ are called the {\em trivially fixed points} of $s$.

An element $[s,x]$ is in the unit space of $\cS\ltimes \beta\cS$ if and only if $[s,x] = [e,x]$ for some idempotent $e$, so we identify the unit space of $\cS\ltimes\beta\cS$ with $\beta\cS$ via $[e,x]\mapsto x$. If $[s,x] = [e,x]$ for $e\in E_\cS$, there is an idempotent $f$ with $x\in \overline {f\cS}$ with $sf = ef$. Thus $sfe = efe = fe$ and so $fe\in\cJ_s$. We also have $[s,x] =[e,x] = [ef,x]$.

It is clear that $x\in F_s$ if and only if $\lambda_s(x) = \mathbf{r}([s,x]) = \mathbf{d}([s,x]) = x$. In summary, we have observed the following lemma:
\begin{lem}\label{ftf}
Let $\mathcal{S}$ be an inverse semigroup. If $s\in\mathcal{S}$, then
\begin{enumerate}
\item $x\in F_s\iff [s,x]\in\textup{Iso}(\cS\ltimes\beta\cS)$;
\item $x\in TF_s\iff [s,x]\in(\cS\ltimes\beta\cS)^{(0)}$;
\item $TF_s\subseteq F_s$.
\end{enumerate}
Consequently, $\cS\ltimes\beta\cS$ is principal if and only if $TF_s = F_s$ for all $s\in \cS$. 
\end{lem}

A key dynamical fact for our purposes is the following consequence of Frol\'ik:

\begin{lem}\label{frolik}
Let $s\in\mathcal{S}$. If $\lambda_s\xi = \xi$ for some $\xi\in \overline{s^*\mathcal{S}}$, then there is a nonempty subset $A\subseteq s^*\mathcal{S}$ such that $A\in\xi$ and $\lambda_s\eta = \eta$, for all $\eta\in \overline{A}$. In particular, $\textup{Iso}(\mathcal{S}\ltimes\beta\mathcal{S})$ is nonempty and open.
\end{lem}
\vspace{-3mm}
\begin{proof}
The result is plain if $\xi$ is a principal ultrafilter, so we may assume that $\xi$ is a free ultrafilter (and therefore that $s^*\mathcal{S}$ is infinite.) Let $B\subset s^*\mathcal{S}$ be an infinite subset with infinite complement. Either $B\in \xi$ or $s^*\mathcal{S}\backslash B\in\xi$; without loss of generality, take $B\in\xi$. Because $\lambda_s: s^*\mathcal{S}\to s\mathcal{S}$ is a bijection, $sB\subset s\mathcal{S}$ is also an infinite subset with infinite complement. Because $\mathcal{S}\backslash B$ and $\mathcal{S}\backslash sB$ are both countably infinite, we may extend the map $\lambda_s: B\to sB$ to a homeomorphism $f: \beta\mathcal{S}\to \beta\mathcal{S}$ with the property that $f(\xi) = \xi$. By Frol{\'i}k's theorem (\cite{Frolik1971}, Theorem 3.1), there is a basic open neighbourhood $U = \overline{C}$ of $\xi$ ($C\subseteq\mathcal{S}$) such that $f(\eta) = \eta$ for all $\eta\in U$. The result follows after taking $A = B\cap C$.
\end{proof}
We can now derive a nice form for the fixed point set of a given $s\in\cS$.
\begin{lem}\label{lem:Fs_form}
    If $\cS$ is an inverse semigroup and $s\in \cS$, then
    \begin{equation}
        F_s = \overline{\bigcup_{e\in \cJ_s}e\cS}.
    \end{equation}
\end{lem}
\begin{proof}
    Suppose that there is some $\xi\in \overline{s^*\mathcal{S}}$ such that $\lambda_s\xi = \xi$. By Lemma~\ref{frolik}, there exists $A\subseteq s^*\cS$ with $A\in\xi$ such that $U_A = \overline{A}$ is fixed by $\lambda_s$. In particular, we have $sa = a$, for all $a\in A$, so $saa^* = aa^*$ for all $a\in A$. Said differently, $aa^* \in \cJ_s$, for all $a\in A$. Thus,
\[
A\subseteq \bigcup_{a\in A}a\mathcal{S} = \bigcup_{a\in A}aa^*\mathcal{S}\subseteq \bigcup_{e\in \cJ_s}e\mathcal{S} \implies \bigcup_{e\in \cJ_s}e\mathcal{S}\in \xi,
\]
because $\xi$ is a filter and filters are upwards closed. Thus $\xi\in \overline{\bigcup_{e\in \cJ_s}e\mathcal{S}}$, and so $F_s \subseteq \overline{\bigcup_{e\in \cJ_s}e\mathcal{S}}$. Reverse containment follows from the continuity of $\lambda_s$, completing the proof.
\end{proof}
\begin{definition}
    Let $\cG$ be a topological groupoid. We say that 
    \begin{enumerate}
        \item $\cG$ is {\em principal} if $g\in \cG$ and $\mathbf{d}(g) = \mathbf r(g)$ implies $g\in \cG^{(0)}$,
        \item $\cG$ is {\em effective} if \textup{Iso}$(\cG)^\circ = \cG^{(0)}$, and
        \item $\cG$ is {\em essentially principal} if $\{x\in \cG^{(0)} : \cG_x^x= \{x\}\}$ is dense in $\cG^{(0)}$.
    \end{enumerate}
\end{definition}

We point out to the reader that some authors use the term \textit{topologically free} to mean what we have defined here as \textit{essentially principal}.

\begin{prop}
    Let $s\in \cS$ and suppose $[s,x]\in\textup{Iso}(\cS\ltimes\beta\cS)$ for some principal filter $x\in s^*\cS$. Then $[s,x]\in (\cS\ltimes\beta\cS)^{(0)}$. In particular, $\cS\ltimes\beta\cS$ is always essentially principal.
\end{prop}
\begin{proof}
    We have that $sx = x$ and so $sxx^* = xx^*$. Since $x\in x\mathcal{S} = xx^*\cS$, we have that $xx^*\cS\in \{x\}^{\subseteq}$. Thus $[s,x]= [sxx^*, x] = [xx^*,x]\in (\cS\ltimes\beta\cS)^{(0)}$. Because principal filters are dense in $(\cS\ltimes\beta\cS)^{(0)}$, $\cS\ltimes\beta\cS$ is essentially principal.
\end{proof}

\begin{theo}\label{pr:principal}
For an inverse semigroup $\mathcal{S}$, the following conditions are equivalent:
\begin{enumerate}
\item $\mathcal{S}\ltimes\beta\mathcal{S}$ is Hausdorff;
\item $\mathcal{S}\ltimes\beta\mathcal{S}$ is principal;
\item $\mathcal{S}\ltimes\beta\mathcal{S}$ is effective.
\end{enumerate}
\end{theo}
\vspace{-3mm}
\begin{proof}
$(1)\implies(2)$ By \cite{EP16}, $\mathcal{S}\ltimes\beta\mathcal{S}$ is Hausdorff if and only if $TF_s = \bigcup_{e\in \cJ_s}\overline{e\mathcal{S}}$ is closed relative to $\overline{s^*\mathcal{S}}$. Because $\overline{s^*\mathcal{S}}$ is already closed, this is equivalent to $TF_s$ being closed. Lemma~\ref{lem:Fs_form} says that $F_s$ is the smallest closed set that contains $e\mathcal{S}$ for all $e\in \cJ_s$, so $F_s\subseteq TF_s$, and hence $F_s = TF_s$. Because every fixed point is trivially fixed, it follows from Lemma \ref{ftf} that $\mathcal{S}\ltimes\beta\mathcal{S}$ is principal.

$(2)\implies (1)$ if $\cG$ is not Hausdorff, we can find a net $(\g_i)$ in $\cG$ with two distinct limits $\g, \g'$. Because the range and source maps are continuous, we have $r(\g) = r(\g')$ and $s(\g) = s(\g')$. Because $\g\neq\g'$, $\cG$ is not principal, and so by the contrapositive we are done.

$(2)\iff (3)$ follows from Lemma \ref{frolik}.
\end{proof}

In light of Theorem~\ref{pr:principal}, the key structural question one asks is \textit{when is $\cS\ltimes\beta\cS$ Hausdorff}? Below is a criterion that produces examples.

\begin{theo}\label{th:Hausdorff_equality}Let $\cS$ be an inverse semigroup. Then the following conditions are equivalent:
\begin{enumerate}
    \item $\cS\ltimes\beta\cS$ is Hausdorff;
    \item for every $s\in\cS$, there exists a finite set $F\subseteq \cJ_s$ such that $\bigcup_{f\in F}f\cS = \bigcup_{e\in\cJ_s}e\cS$;
    \item for every $s\in \cS$, there exists a finite set $F\subseteq \cJ_s$ such that $\cJ_s = F^\geqslant$.
\end{enumerate} 
\end{theo}
\begin{proof}
    $(1)\implies(2)$  By Theorem~\ref{pr:principal}, if $\cS\ltimes \beta\cS$ is Hausdorff, then it is principal and so $F_s = TF_s$ for all $s\in \cS$. We always have $TF_s = \bigcup_{e\in \cJ_s}\overline{e\mathcal{S}}$ and Lemma~\ref{lem:Fs_form} implies $F_s = \overline{\bigcup_{e\in \cJ_s}e\mathcal{S}}$. Thus for all $s\in \cS$ we have $\overline{\bigcup_{e\in \cJ_s}e\mathcal{S}}= \bigcup_{e\in \cJ_s}\overline{e\mathcal{S}}$. This set is compact and each $\overline{e\cS}$ is open, so there is a finite $F\subseteq \cJ_s$ such that 
\[
\overline{\bigcup_{e\in \cJ_s}e\mathcal{S}}= \bigcup_{e\in \cJ_s}\overline{e\mathcal{S}} = \bigcup_{f\in F}\overline{f\mathcal{S}} = \overline{\bigcup_{f\in F}f\mathcal{S}},
\]
where the first equality is the expression $F_s = TF_s$ and the last equality follows as $F$ is finite. Because elements of $\cS$ are isolated in $\beta\cS$, this implies $\bigcup_{f\in F}f\cS = \bigcup_{e\in\cJ_s}e\cS$.

$(2)\implies (3)$ For each $e\in \cJ_s$, there exists $f\in F$ such that $e\in f\cS$ so we can write $e = ft$ for some $t\in \cS$. Thus $fe = fft = ft = e$, implying $e\leqslant f$. Conversely, if $e\leqslant f\in F\subseteq\cJ_s$, $se = sfe = fe = e$ and so $e\in \cJ_s$.

$(3)\implies (2)$ If $e\in \cJ_s$, then by assumption, $e\leqslant f$ for some $f\in F$; so $e\in f\cS$. Thus, we have $\bigcup_{f\in F}f\cS = \bigcup_{e\in\cJ_s}e\cS$.

$(2)\implies (1)$ We have
\[
F_s = \overline{\bigcup_{e\in \cJ_s}e\mathcal{S}} = \overline{\bigcup_{f\in F}f\mathcal{S}} = \bigcup_{f\in F}\overline{f\mathcal{S}}\subseteq \bigcup_{e\in \cJ_s}\overline{e\mathcal{S}} = TF_s,
\]
where the middle equality follows because the union in question is finite. This implies $F_s = TF_s$. By Theorem~\ref{pr:principal}, $\cS\ltimes\beta\cS$ is Hausdorff.
\end{proof}

Two elements $s,t \in\cS$ are called {\em compatible} if both $s^*t$ and $st^*$ are idempotents. Notice that idempotents are always compatible. The inverse semigroup $\cS$ is called {\em complete} if whenever $A\subseteq\cS$ is a set of pairwise compatible elements, then $A$ has a least upper bound in $\cS$ (in the natural partial order); in this case the least upper bound is denoted $\bigvee_{a\in A}a$. An inverse semigroup is called {\em infinitely distributive} if whenever $\bigvee_{a\in A}a\in\cS$ and $s\in \cS$, we have 
\[
s\left(\bigvee_{a\in A}a\right) = \bigvee_{a\in A}sa\hspace{1cm} \text{and}\hspace{1cm}\left(\bigvee_{a\in A}a\right)s = \bigvee_{a\in A}as.
\]
An inverse semigroup is an \textit{abstract pseudogroup} if it is a complete, infinitely distributive inverse monoid (see \cite{lawson}). This notion is not to be confused with other inequivalent definitions of pseudogroup, such as a pseudogroup in the sense of Renault (cf. \cite{cartan}). The symmetric inverse monoid $\mathcal{I}(X)$ is an abstract pseudogroup. An inverse semigroup with zero  $\mathcal{S}$ is \textit{$E^*$-unitary} if $s\in\mathcal{S}$ and $\cJ_s\neq\{0\}$ implies $s\in E_S$. Graph inverse semigroups and free inverse semigroups are $E^*$-unitary \cite{L}. The next result is an immediate consequence of Theorem \ref{th:Hausdorff_equality}.

\begin{cor}\label{hausdorffexample}
If $\mathcal{S}$ is an abstract pseudogroup or an $E^*$-unitary inverse semigroup, then $\cS\ltimes\beta\cS$ is Hausdorff.
\end{cor}
\begin{proof}
    If $\cS$ is an abstract pseudogroup and $s\in \cS$, then $\cJ_s$ is pairwise compatible and hence $e:= \bigvee_{a\in\cJ_s}a\in \cS$. We also have that $se = \bigvee_{a\in\cJ_s}sa = \bigvee_{a\in\cJ_s}a = e$ by distributivity, so $e\in \cJ_s$. The finite set $F = \{e\}$ satisfies Theorem~\ref{th:Hausdorff_equality}(3).
    
    Every transformation groupoid of an $E^*$-unitary inverse semigroup is Hausdorff by \cite{EP16}. Alternatively, observe that if $\mathcal{J}_s$ is nonzero, then $s\in \mathcal{J}_s$, and thus $\mathcal{J}_s = \{s\}^{\geqslant}$.
\end{proof}

The main result of \cite{starling} says that the action of $\cS$ on its spectrum $\widehat{E}_0$ (the space of nonzero semicharacters on $E_{\mathcal{S}}$) is amenable if and only if every action of $\cS$ is amenable. Our last result in this section has a similar flavour, and provides some clear obstructions to the Hausdorffness of $\cS\ltimes \beta\cS$.
\begin{definition}
    Say that an action $(\cS, \cX, \alpha)$ is {\em clopen} if $D^\alpha_e$ is clopen for each $e\in E_{\mathcal{S}}$. 
\end{definition}
\begin{cor}\label{cor:clopen_action}
    For an inverse semigroup $\cS$, the following conditions are equivalent:

    \begin{enumerate}
        \item $\cS\ltimes \beta\cS$ is Hausdorff;
        \item $\cS \ltimes \cX$ is Hausdorff for every clopen action $(\cS, \cX, \alpha)$.
    \end{enumerate}

\end{cor}

\begin{proof}
We need only check (1)$\implies$(2), as the action of $\cS$ on $\beta\cS$ is clopen.
   
Let $s\in\mathcal{S}$. By assumption, $\alpha_s: D^\alpha_{s^*s}\to D^\alpha_{ss^*}$ a homeomorphism between clopen sets. Writing $TF^\alpha_s$ for the trivially fixed points for the action $\alpha$, we have
\[
TF^\alpha_s = \bigcup_{e\in \cJ_s}D^\alpha_e.
\]
Applying Theorem~\ref{th:Hausdorff_equality}, we can find a finite set $F\subseteq \cJ_{s}$ such that, for all $e\in \cJ_s$, there exists $f\in F$ with $e\leqslant f$. Since $D_e^{\alpha}\subseteq D_f^{\alpha}$ whenever $e\leqslant f$, we have
\[
TF^\alpha_s = \bigcup_{e\in \cJ_s}D^\alpha_e = \bigcup_{f\in F}D^\alpha_f.
\]
By assumption, each $D^\alpha_f$ is closed. Since $F$ is finite, $TF^\alpha_s$ is closed, so by \cite{EP16}*{Theorem~3.15}, $\cS\ltimes\cX$ is Hausdorff.
\end{proof}

 \begin{cor}
     If $\cS\ltimes\beta\cS$ is Hausdorff, then Exel's tight groupoid of $\cS$, Exel's universal groupoid of $\cS$, and Paterson's universal groupoid of $\cS$ are all Hausdorff. 
 \end{cor}

For more details, Exel's universal and tight groupoid constructions can be found in \cite{EXEL}; Paterson's universal groupoid construction can be found in \cite{P}. Negating Corollary~\ref{cor:clopen_action} then shows that we can find many examples where $\cS\ltimes\beta\cS$ is not Hausdorff (and hence not principal), namely any example where the tight groupoid is not Hausdorff. See \cite{CEPSS}*{Example~5.6} for an example arising from the Grigorchuk group.

We thank the anonymous referee for the following example, which shows that Corollary~\ref{cor:clopen_action} does not hold for non-clopen actions. Let $\mathcal{S} = \{0, s, 1\}$, where $0$ and $1$ act as zero and identity elements of $\mathcal{S}$, and $s^2 = 1$. Since $\mathcal{S}$ is finite, $\mathcal{S}\ltimes\beta\mathcal{S}$ is  discrete, and in particular Hausdorff. However, $\mathcal{S}$ acts trivially on $[0, 1]$, with $D_1 = D_s = [0, 1]$ and $D_0 = [0, 1)$; the resulting action groupoid $\mathcal{S}\ltimes[0, 1]$ is not Hausdorff, because the elements $[s, 1]$ and $[0, 1]$ cannot be separated by disjoint neighbourhoods.

\section{Amenability and the uniform Roe algebra}

In this last, self-contained section, we clarify the relationship between several crossed product constructions associated to the Stone-\v Cech action $\lambda: \mathcal{S}\acts \beta\mathcal{S}$, and relate the amenability of the Stone-\v Cech action to the nuclearity and exactness of these constructions. 

We begin by collecting several definitions and existing results.

\subsection*{The uniform Roe algebra $\mathcal{R}_{\mathcal{S}}$.} Given an inverse semigroup $\mathcal{S}$, we define the \textit{left regular representation} $\lambda: \mathcal{S}\to \mathcal{B}(\ell^2(\mathcal{S}))$ by
\begin{align*}
\lambda_s\delta_t = \begin{cases}
\delta_{st} & t\in s^*\mathcal{S};\\
0 & \text{otherwise.}
\end{cases}
\end{align*}
The C$^*$-algebra generated by $\{\lambda_s: s\in\mathcal{S}\}$ in $\mathcal{B}(\ell^2(\mathcal{S}))$ is called the \textit{reduced \textup{C}$^*$-algebra} of $\mathcal{S}$, denoted by $\textup{C}^*_r(\mathcal{S})$.

The C$^*$-algebra generated by $\{M_f\lambda_s: s\in\mathcal{S}, f\in\ell^{\infty}(\mathcal{S})\}$, where $\ell^{\infty}(\mathcal{S})\ni f\mapsto M_f\in \mathcal{B}(\ell^2(\mathcal{S}))$ is the diagonal embedding, is called the \textit{uniform Roe algebra} of $\mathcal{S}$, denoted by $\mathcal{R}_{\mathcal{S}}$. Note that we have an inclusion $\textup{C}^*_r(\mathcal{S})\subset\mathcal{R}_{\mathcal{S}}$.

For the convenience of the reader, we recall the following terminology from \cite{lledomartinez}: an inverse semigroup $\mathcal{S}$ has the property (FL) if there is a symmetric, generating subset $K\subset\mathcal{S}$, a natural number $N\in\mathbb{N}$, and a finite subset $K_0\subset K$ such that: for all $s, t\in\mathcal{S}$ with $s^*s = t^*t$ and $t\in Ks$, we have $y\in (K_1)^nx$, for some $n\leq N$. Clearly, every finitely generated inverse semigroup has (FL); on the other hand, the free inverse semigroup on countably many generators has (FL) (see \cite{lledomartinez}*{Proposition~3.18}.) The relevance of this technical condition is in the following theorem, which generalizes \cite{ozawa}*{Theorem 3}:

\begin{theo}[\cite{lledomartinez}*{Theorem~4.18}]
If $\mathcal{S}$ is an inverse semigroup with \textup{(FL)}, then the following conditions are equivalent:
\begin{enumerate}
\item $\textup{C}^*_r(\mathcal{S})$ is exact;
\item $\mathcal{R}_{\mathcal{S}}$ is exact;
\item $\mathcal{R}_{\mathcal{S}}$ is nuclear.
\end{enumerate}
\end{theo}

\subsection*{$\mathcal{R}_{\mathcal{S}}$ as the image of a covariant representation.} An \textit{action} of an inverse semigroup $\mathcal{S}$ on a unital C$^*$-algebra $\mathcal{A}$ is a triple $(\mathcal{S}, \mathcal{A}, \alpha)$, where $\alpha: \mathcal{S}\to \mathcal{I}(\mathcal{A})$ is a semigroup homomorphism such that, for each $s\in\mathcal{S}$,
\begin{enumerate}
\item The domain and codomain of $\alpha_s$ is a closed, two-sided ideal in $\mathcal{A}$;
\item The partial bijection $\alpha_s$ is a $*$-homomorphism.
\end{enumerate}
Given an idempotent $e\in E_{\mathcal{S}}$, we set $J_e = \text{dom}(\alpha_e)$. We observe that $\alpha_s$ is a $*$-isomorphism of $J_{s^*s}$ onto $J_{ss^*}$.

Gelfand duality induces a correspondence between actions of $\mathcal{S}$ on compact Hausdorff spaces and actions of $\mathcal{S}$ on commutative unital C$^*$-algebras. Given an action $\alpha: \mathcal{S}\acts\mathcal{X}$, one obtains an inverse semigroup action $\hat{\alpha}: \mathcal{S}\acts C(\mathcal{X})$ by setting $J_{ss^*} = \{f\in C(\mathcal{X}): f|_{\mathcal{X}\backslash D_{ss^*}}\equiv 0\}$ and defining $\hat{\alpha}: J_{s^*s}\to J_{ss^*}$ by $\hat{\alpha}_sf = f\circ\alpha_{s^*}$ ($f\in J_{s^*s}$). Our main interest in this section is the following example:

\begin{ex}[\cite{lledomartinez}*{Proposition~2.4}]\label{ellinfty}
Let $\mathcal{A} = \ell^{\infty}(\mathcal{S})$. For each $e\in E_{\mathcal{S}}$, define $J_e = \{f\in \mathcal{A}: \text{supp}(f)\subseteq e\mathcal{S}\}$. For each $s\in\mathcal{S}$, define $\alpha_s: J_{s^*s}\to J_{ss^*}$ by $[\alpha_sf](t) = f(s^*t)$. The mapping $s\in\mathcal{S}\mapsto \alpha_s\in \mathcal{I}(\mathcal{S})$ is an action of $\mathcal{S}$ on $\mathcal{A}$.
\end{ex}
Using the canonical isomorphism $C(\beta \mathcal{S})\cong \ell^{\infty}(\mathcal{S})$, the reader can easily verify that the action $\alpha: \mathcal{S}\acts \ell^{\infty}(\mathcal{S})$ is the contravariant action $\hat{\lambda}$ obtained from the Stone-\v Cech action $\lambda:\mathcal{S}\acts \beta\mathcal{S}$ defined in Example \ref{sc}.

A \textit{covariant representation} of an inverse semigroup action $\alpha: \mathcal{S}\acts\mathcal{A}$ on a unital C$^*$-algebra $\mathcal{A}$ is a triple $(\mathcal{H}, \pi, \sigma)$, where $\mathcal{H}$ is a Hilbert space, $\pi: \mathcal{A}\to \mathcal{B}(\mathcal{H})$ is a $*$-homomorphism, and $\sigma: \mathcal{S}\to \mathcal{B}(\Hi)$ is a semigroup homomorphism (under composition) such that, for all $e\in E_{\mathcal{S}}$, $s\in\mathcal{S}$, and $a\in J_{s^*s}$,
\begin{enumerate}
\item $\sigma_{s^*} = \sigma_s^*$;
\item $\sigma_{s^*}\pi(a)\sigma_s = \pi(\alpha_sa)$, and
\item (Non-degeneracy) $\overline{\pi(J_{e})\mathcal{H}}=\sigma_e\mathcal{H}$.
\end{enumerate}

In \cite{lledomartinez}*{Section 2}, the following triple $(\Hi, \pi, \sigma)$ is considered for the action $\mathcal{S}\acts\ell^{\infty}(\mathcal{S})$: $\Hi = \ell^2(\mathcal{S})\otimes\ell^2(\mathcal{S})$, and the maps $\pi: \ell^{\infty}(\mathcal{S})\to \mathcal{B}(\Hi)$ and $\sigma: \mathcal{S}\to\mathcal{B}(\Hi)$ are determined by the formulas
\begin{align*}
[\pi(f)](\delta_s\otimes \delta_t) = \begin{cases}
f(ts)\delta_s\otimes \delta_t & s\in t^*\mathcal{S}\\
0 & \text{otherwise}
\end{cases}
\end{align*}
and
\begin{align*}
[\sigma_r](\delta_s\otimes\delta_t) = \begin{cases}
\delta_{s}\otimes\delta_{rt} & t\in r^*\mathcal{S}\\
0 & \text{otherwise.}
\end{cases}
\end{align*}
It is easily verified that the triple $(\Hi, \pi, \sigma)$ satisfies the first two conditions of a covariant representation. The non-degeneracy condition (3) is immediate, since $\pi(J_e)\mathcal{H} = \pi(\ell^{\infty}(e\mathcal{S}))\mathcal{H} = \ell^2(\mathcal{S})\otimes \ell^2(e\mathcal{S}) = \sigma_e\mathcal{H}$ (where the embeddings $\ell^{p}(e\mathcal{S})\subset \ell^{p}(\mathcal{S})$ for $p\in \{2, \infty\}$ are obtained by setting $f(t) = 0$ for $t\not\in e\mathcal{S}$, $f\in \ell^p(e\mathcal{S})$.) Thus, the triple $(\Hi, \pi, \sigma)$ forms a covariant representation for the action $\mathcal{S}\acts \ell^{\infty}(\mathcal{S})$. We denote the C$^*$-subalgebra generated by $(\Hi, \pi, \sigma)$ by $\ell^{\infty}(\mathcal{S})\rtimes_{rr}\mathcal{S}\subset \mathcal{B}(\ell^2(\mathcal{S})\otimes\ell^2(\mathcal{S}))$. 

\textbf{Note:} in \cite{lledomartinez}, the C$^*$-algebra $\ell^{\infty}(\mathcal{S})\rtimes_{rr}\mathcal{S}$ is denoted by ``$\ell^{\infty}(\mathcal{S})\rtimes_{r}\mathcal{S}$.'' We will reserve this notation for the reduced crossed product of Buss-Exel-Meyer \cite{bussexel1}, defined shortly. For now, we record the following result:

\begin{theo}[\cite{lledomartinez}*{Theorem 2.5}]
If $\mathcal{S}$ is an inverse semigroup, then the mapping $\pi(f)\sigma_s\mapsto M_f\lambda_s$ determines an isomorphism
\begin{align*}
\ell^{\infty}(\mathcal{S})\rtimes_{rr}\mathcal{S}\cong\mathcal{R}_{\mathcal{S}}.
\end{align*}
\end{theo}

\subsection{Reduced groupoid C$^*$-algebras} We will shortly compare the above constructions to the reduced crossed product of Buss, Exel, and Meyer, which is defined in full generality in \cite{bussexel1}. Our working definition for this construction in the present note will be from a groupoid model. To keep this section brief, we will only summarize the key facts we need; for details on the construction of the reduced C$^*$-algebra of general \'etale groupoids, see \cite{bussexel1}. Recall that if $\mathcal{G}$ is an \'etale groupoid, then the $*$-vector space $B_c(\mathcal{G})$ of compactly supported Borel functions on $\mathcal{G}$ is equipped with a $*$-algebra structure under the convolution product
\begin{align*}
f*g(\gamma) = \sum_{\alpha\beta = \gamma}f(\alpha)g(\beta).
\end{align*}

The set $\mathscr{C}_c(\mathcal{G}) := \text{span}(\bigcup_{U\in\mathscr{B}}C_c(U))\subset B_c(\mathcal{G})$, where $\mathscr{B}$ is the set of open bisections in $\mathcal{G}$, is a $*$-subalgebra. For each $x\in \mathcal{G}^{(0)}$, let $\rho_x: \mathscr{C}_c(\mathcal{G})\to \mathcal{B}(\ell^2(\mathcal{G}_x))$ denote the $*$-representation of $\mathscr{C}_c(\mathcal{G})$ given by
	\[
	\rho_x(f)\delta_\g = \sum_{\alpha\in\mathcal{G}_{\mathbf{r}(x)}}f(\alpha)\delta_{\alpha\g}, \hspace{1cm} f\in \mathscr{C}_c(\mathcal{G}), \g\in \mathcal{G}_x.
	\]
	The {\em reduced \textup{C$^*$}-algebra} of $\mathcal{G}$, denoted $\textup{C}_r^*(\mathcal{G})$, is the completion of the image of $\mathscr{C}_c(\mathcal{G})$ under the representation $\bigoplus_{x\in \mathcal{G}^{(0)}}\rho_x$. By \cite{renault1980}*{Proposition 1.11}, this representation is faithful. In particular, if $\mathcal{G}^{(0)}$ is compact, then we have a canonical inclusion $C(\mathcal{G}^{(0)})\subset \textup{C}^*_r(\mathcal{G})$.

Let $A$ be a unital C$^*$-algebra, and $B\subset A$ a unital C$^*$-subalgebra. Recall that a \textit{weak conditional expectation} $\mathbb{E}: A\to B$ is a unital completely positive map $\mathbb{E}: A\to B^{**}$ such that $\mathbb{E}|_{B} = \text{id}_B$ and $\mathbb{E}(bab') = b\mathbb{E}(a)b'$, for all $a\in A$, $b, b'\in B$.

\begin{theo}\label{wce}\cite{renault1980}*{Proposition 4.8}
Let $\mathcal{G}$ be an \'etale groupoid with compact unit space $\mathcal{G}^{(0)}$.
\begin{enumerate}
\item There is a faithful weak conditional expectation $\mathbb{E}: \textup{C}^*_r(\mathcal{G})\to C(\mathcal{G}^{(0)})^{**}$, which has range $C(\mathcal{G}^{(0)})$ if and only if $\mathcal{G}$ is Hausdorff.
\item If $\mathcal{G}$ is Hausdorff, then $\mathscr{C}_c(\mathcal{G}) = C_c(\mathcal{G})$, and for all $f\in C_c(\mathcal{G})$ we have $\mathbb{E}(f) = f|_{\mathcal{G}^{(0)}}.$
\end{enumerate}
\end{theo}

Now let $\mathcal{S}$ be a unital inverse semigroup acting on a compact Hausdorff space $\mathcal{X}$, and consider $\mathcal{G} = \mathcal{S}\ltimes\mathcal{X}$. Given $s\in \mathcal{S}$ and $f\in J_{ss^*}$, we define an element $fs\in \mathscr{C}_c(\mathcal{G})$ by
\begin{align*}
fs([t, x]) = \begin{cases}
f([t, x]) & \text{if }[t, x] = [s, x];\\
0 & \text{otherwise.}
\end{cases}
\end{align*}

\begin{theo}\label{algebraiccrossedproduct}
Let $\mathcal{S}$, $\mathcal{X}$, and $\mathcal{G}$ be as above, and let $C(\mathcal{X})\rtimes_r \mathcal{S}$ be the reduced crossed product of $(\mathcal{S}, \mathcal{X})$ in the sense of \cite{bussexel1}.

\begin{enumerate}
\item \cite{EXEL}*{Proposition 3.10} The functions $fs$ $(s\in\mathcal{S}, f\in J_{ss^*})$ span $\mathscr{C}_c(\mathcal{G})$.
\item \cite{EXEL}*{Theorem 8.5} If $(\Hi, \pi, \sigma)$ is a covariant representation of $(\mathcal{S}, \mathcal{X})$, then the formula
\begin{align*}
(\pi\times\sigma)(fs) = \pi(f)\sigma_s
\end{align*}
determines a unique $*$-representation $\pi\times\sigma: \mathscr{C}_c(\mathcal{G})\to \mathcal{B}(\Hi)$.
\item\cite{bussexel}*{Theorem 4.11} $\textup{C}^*_r(\mathcal{G}) \cong C(\mathcal{X})\rtimes_r \mathcal{S}$.
\end{enumerate}
\end{theo}

\begin{prop}\label{completion}
Let $\mathcal{S}$ be an inverse semigroup. If $\mathcal{S}\rtimes\beta\mathcal{S}$ is Hausdorff, then:
\begin{enumerate}
\item The $*$-homomorphism $\pi\times\sigma: C_c(\mathcal{S}\rtimes\beta\mathcal{S})\to \ell^{\infty}(\mathcal{S})\rtimes_{rr}\mathcal{S}$ is faithful. 
\item The faithful conditional expectation $\mathbb{E}: C_c(\mathcal{S}\rtimes\beta\mathcal{S})\to \ell^{\infty}(\mathcal{S})$ descends to a faithful conditional expectation on $\ell^{\infty}(\mathcal{S})\rtimes_{rr}\mathcal{S}$.
\end{enumerate}
\end{prop}

\begin{proof}
We prove (2) first. This follows from Theorem \ref{algebraiccrossedproduct}(1) because
\begin{align*}
[\mathbb{E}(fs)]\delta_t = \langle M_f\lambda_s\delta_t, \delta_t\rangle,
\end{align*}
for all $s, t\in\mathcal{S}$ and $f\in J_{ss^*}$, and because the mapping 
\begin{align*}
\mathbb{E}':\mathcal{B}(\ell^2(\mathcal{S}))\ni T\mapsto [\mathcal{S}\ni t\mapsto \langle T\delta_t, \delta_t\rangle\in\C]\in \ell^{\infty}(\mathcal{S})
\end{align*}
is a faithful conditional expectation.

(1) follows because if $a\in C_c(\mathcal{S}\rtimes\beta\mathcal{S})$, then $(\pi\times\sigma)(a) = 0\Leftrightarrow(\pi\times\sigma)(a^*a) = 0\Leftrightarrow\mathbb{E}'((\pi\times\sigma)(a^*a)) = 0\Leftrightarrow\mathbb{E}(a^*a) = 0\Leftrightarrow a = 0$.
\end{proof}

\begin{cor}\label{isomorphism}
If $\mathcal{S}$ is an inverse semigroup, then the following conditions are equivalent:
\begin{enumerate}
\item $\mathcal{S}\ltimes\beta\mathcal{S}$ is Hausdorff;
\item The map $fs\mapsto M_f\lambda_s$ extends to an isomorphism $\ell^{\infty}(\mathcal{S})\rtimes_{r}\mathcal{S}\cong\ell^{\infty}(\mathcal{S})\rtimes_{rr}\mathcal{S}$.
\end{enumerate}
\end{cor}
\vspace{-3mm}
\begin{proof}
If $\mathcal{S}\ltimes\beta\mathcal{S}$ is Hausdorff, then by Proposition \ref{completion}, both $\mathcal{A}_1 = \ell^{\infty}(\mathcal{S})\rtimes_{rr}\mathcal{S}$ and $\mathcal{A}_2 = \ell^{\infty}(\mathcal{S})\rtimes_{r}\mathcal{S}$ are completions of $C_c(\mathcal{S}\rtimes\beta\mathcal{S})$ which admit faithful extensions of the conditional expectation $\mathbb{E}$. It follows that $\mathcal{A}_1$ and $\mathcal{A}_2$ act faithfully on the left multiplication Hilbert $\ell^{\infty}(\mathcal{S})$-modules $L^2(\mathcal{A}_1, \mathbb{E})$ and $L^2(\mathcal{A}_2, \mathbb{E})$ (respectively), and that the identity map on \hbox{$C_c(\mathcal{S}\rtimes\beta\mathcal{S})$} extends to a unitary $\mathcal{U}: L^2(\mathcal{A}_1, \mathbb{E})\to L^2(\mathcal{A}_2, \mathbb{E})$, which implements an isomorphism between $\mathcal{A}_1$ and $\mathcal{A}_2$.

Conversely, if $\mathcal{S}\ltimes \beta\mathcal{S}$ is not Hausdorff, then (by Theorem \ref{wce}(1)) the weak conditional expectation $\mathbb{E}: \mathcal{A}_2\to \ell^{\infty}(\mathcal{S})^{**}$ takes values outside of $\ell^{\infty}(\mathcal{S})$, while the conditional expectation $\mathbb{E}': \mathcal{A}_1\to \ell^{\infty}(\mathcal{S})$ from the proof of Proposition \ref{completion} does not. If the mapping $fs\mapsto M_f\lambda_s$ extended to an isomorphism between $\mathcal{A}_1$ and $\mathcal{A}_2$, then by Proposition \ref{completion}, we would have $\mathbb{E}' = \mathbb{E}$, a contradiction.
\end{proof}

We now arrive at the main theorem of this section.

\begin{theo}
Suppose $\mathcal{S}$ is a countable inverse semigroup with the property \textup{(FL)} of Lled\'o and Mart\'inez. If $\mathcal{S}\ltimes\beta\mathcal{S}$ is Hausdorff, then the action $\mathcal{S}\acts\beta\mathcal{S}$ is amenable \textup{(}in the sense of \cite{starling}\textup{)} if and only if $\textup{C}^*_r(\mathcal{S})$ is exact.
\end{theo}

\begin{proof}
By \cite{lledomartinez}*{Theorem 4.18}, $\textup{C}^*_r(\mathcal{S})$ is exact if and only if $\mathcal{R}_\mathcal{S}\cong \ell^{\infty}(\mathcal{S})\rtimes_{rr}\mathcal{S}$ is nuclear. By Theorem \ref{isomorphism}, this is equivalent to the nuclearity of $\ell^{\infty}(\mathcal{S})\rtimes_{r}\mathcal{S}\cong \textup{C}^*_r(\mathcal{S}\ltimes\beta\mathcal{S})$. By \cite{AR}*{ Corollary 6.2.14, Theorem 3.3.7}, this is equivalent to the amenability of $\mathcal{S}\ltimes\beta\mathcal{S}$.
\end{proof}

We conclude with a brief discussion. In light of the above theorem, the following question seems natural: What is the relationship between the amenability of $\mathcal{S}\ltimes\beta\mathcal{S}$ and the exactness of $\textup{C}^*_r(\mathcal{S})$ when $\mathcal{S}\ltimes\beta\mathcal{S}$ is non-Hausdorff? We remark that it was shown in \cite{BM}*{Theorem A} and independently in \cite{bghx}*{Theorem F} that, for a (possibly non-Hausdorff) $\sigma$-compact \'etale groupoid $\mathcal{G}$, the amenability of $\mathcal{G}$ is equivalent to the nuclearity of $\textup{C}^*_r(\mathcal{G})$. Since $\mathcal{S}\ltimes\beta\mathcal{S}$ is $\sigma$-compact whenever $\mathcal{S}$ is countable, these results show that the action $\mathcal{S}\acts\beta\mathcal{S}$ is amenable in the sense of \cite{starling} if and only if the Buss-Exel-Meyer reduced crossed product $\ell^{\infty}(\mathcal{S})\rtimes_r \mathcal{S}$ is nuclear.

\nocite{*}
\bibliographystyle{plain}
\bibliography{references}
\end{document}